\documentclass[12pt, reqno]{amsart}

\numberwithin{equation}{section}
\newcommand*{\LatexDef}{.}
\usepackage{kpfonts}

\usepackage{amsthm}

\usepackage{amsmath}

\usepackage{amssymb}

\usepackage{dsfont} 

\usepackage{thmtools}

\usepackage{cancel}

\usepackage{verbatim}

\usepackage[shortlabels]{enumitem}
\setlist[enumerate]{leftmargin=*,font=\upshape,align=parleft,label=(\alph*)}
\setlist[itemize]{leftmargin=*,labelwidth=*}

\usepackage{graphicx}
\usepackage{float}

\usepackage{color}

\usepackage{xspace}

\usepackage{calrsfs}

\DeclareSymbolFont{defaultmathcal}{OMS}{zplm}{m}{n}
\DeclareSymbolFontAlphabet{\mathcal}{defaultmathcal}

\DeclareSymbolFont{handwritten}{OMS}{rsfs}{m}{n}
\DeclareSymbolFontAlphabet{\handcal}{handwritten}

\usepackage[alphabetic, nobysame]{amsrefs}

\definecolor{darkgreen}{RGB}{54,124,50}
\usepackage{hyperref}
\hypersetup{
	pdfstartview={XYZ null null 1.00}, 
	pdfpagemode=UseNone, 
	colorlinks,
	breaklinks,
	linkcolor=blue,
	urlcolor=blue, 
	anchorcolor=blue,
	citecolor=darkgreen,
}

\usepackage[capitalise]{cleveref}

\usepackage{xparse} 

\usepackage{geometry}
\usepackage{microtype}

\ExplSyntaxOn
\NewDocumentCommand{\RN}{m}
{
	\textup{ \int_to_Roman:n { #1 } }
}
\NewDocumentCommand{\rn}{m}
{
	\textup{ \int_to_roman:n { #1 } }
}
\ExplSyntaxOff

\newcommand{\N}{\mathbb{N}}

\newcommand{\0}{\mathbb{\emptyset}}
\newcommand{\w}{\infty}

\newcommand\al{\alpha}

\newcommand{\ga}{\gamma}
\newcommand\de{\delta}

\newcommand\la{\lambda}

\newcommand{\Ga}{\Gamma}
\newcommand{\De}{\Delta}
\newcommand\La{\Lambda}
\newcommand{\Si}{\Sigma}

\renewcommand{\phi}{\varphi}

\DeclareRobustCommand{\rchi}{{\mathpalette\irchi\relax}}
\newcommand{\irchi}[2]{\raisebox{\depth}{$#1\cchi$}}
\let\cchi\chi
\let\chi\rchi

\newcommand{\AC}{\mathcal{A}} 
\newcommand{\BC}{\mathcal{B}}
\newcommand{\CC}{\mathcal{C}}

\newcommand{\NC}{\mathcal{N}}

\newcommand{\SC}{\mathcal{S}}
\newcommand{\TC}{\mathcal{T}}

\newcommand{\YC}{\mathcal{Y}}
\newcommand{\ZC}{\mathcal{Z}}

\newcommand{\BF}{\mathfrak{B}}
\newcommand{\CF}{\mathfrak{C}}

\newcommand{\Aut}{\operatorname{Aut}}

\newcommand{\Pow}{\handcal{P}}

\newcommand{\Stab}{\operatorname{Stab}}

\newcommand{\del}{\partial}

\makeatletter
\def\moverlay{\mathpalette\mov@rlay}

\def\mov@rlay#1#2{\leavevmode\vtop{%
		\baselineskip\z@skip \lineskiplimit-\maxdimen
		\ialign{\hfil$\m@th#1##$\hfil\cr#2\crcr}}}

\newcommand{\charfusion}[3][\mathord]{
	#1{\ifx#1\mathop\vphantom{#2}\fi
		\mathpalette\mov@rlay{#2\cr#3}
	}
	\ifx#1\mathop\expandafter\displaylimits\fi}
\makeatother

\renewcommand{\le}{\leqslant}

\renewcommand{\ge}{\geqslant}

\newcommand{\actson}{\curvearrowright}

\newcommand{\conc}{{^\smallfrown}}

\newcommand{\onto}{\twoheadrightarrow}

\newcommand{\shortiff}{\Leftrightarrow}

\newcommand{\nperp}{\charfusion[\mathrel]{\diagup}{\perp}}

\newcommand*{\defeq}{\mathrel{\vcenter{\baselineskip0.5ex \lineskiplimit0pt \hbox{\scriptsize.}\hbox{\scriptsize.}}}=}

\newcommand*{\defequiv}{\mathrel{\vcenter{\baselineskip0.5ex \lineskiplimit0pt
			\hbox{\scriptsize.}\hbox{\scriptsize.}}}\shortiff}

\DeclareMathSymbol{\lqm}{\mathord}{operators}{``}
\DeclareMathSymbol{\rqm}{\mathord}{operators}{`'}

\newcommand{\set}[1]{\left\{ #1 \right\}}

\newcommand{\rest}[1]{\mathord{|_{#1}}}

\theoremstyle{plain}

\newtheorem{theorem}[equation]{Theorem}
\newtheorem*{theorem*}{Theorem}

\makeatletter
\def\@empty{}
\def\ifemptycredit#1{%
	\def\tmp{#1}%
	\ifx\tmp\@empty%
	\else%
	{~(#1)}%
	\fi%
}
\makeatother

\newenvironment{namedthm*}[2][]{
	\par\medskip\noindent \textbf{#2}\ifemptycredit{#1}\textbf{.}\itshape\xspace
}{}

\crefname{prop}{Proposition}{Propositions}
\newtheorem{prop}[equation]{Proposition}
\newtheorem*{propo*}{Proposition}

\crefname{property}{Property}{Properties}

\newtheorem*{property*}{Property}

\newtheorem{lemma}[equation]{Lemma}
\newtheorem*{lemma*}{Lemma}

\crefname{claimlemma}{Claim}{Claims}

\crefname{cor}{Corollary}{Corollaries}
\newtheorem{cor}[equation]{Corollary}
\newtheorem*{cor*}{Corollary}

\crefname{obs}{Observation}{Observations}
\newtheorem{obs}[equation]{Observation}
\newtheorem*{obs*}{Observation}

\crefname{obss}{Observations}{Observations}

\newtheorem{obss*}{Observations}

\crefname{fact}{Fact}{Facts}

\newtheorem*{fact*}{Fact}

\theoremstyle{definition}

\crefname{defn}{Definition}{Definitions}
\newtheorem{defn}[equation]{Definition}
\newtheorem*{defn*}{Definition}

\newenvironment{defn**}[1][]{\par\medskip\noindent \textbf{Definition\xspace#1.}\xspace}{}

\crefname{question}{Question}{Questions}

\newtheorem*{question*}{Question}

\crefname{conj}{Conjecture}{Conjectures}

\newtheorem*{conj*}{Conjecture}

\crefname{example}{Example}{Examples}
\newtheorem{example}[equation]{Example}
\newtheorem*{example*}{Example}

\crefname{examples.plain}{Examples}{Examples}
\newtheorem{examples.plain}[equation]{Examples}
\newtheorem*{examples.plain*}{Examples}

\theoremstyle{remark}

\crefname{remark}{Remark}{Remarks}
\newtheorem{remark}[equation]{Remark}
\newtheorem*{remark*}{Remark}

\newenvironment{remarklike*}[2][]{\par\medskip\noindent \textit{#2}#1\textbf{.}\rmfamily\xspace}{\smallskip}

\crefname{claim+}{Claim}{Claims}
\newtheorem{claim+}[equation]{Claim}

\crefname{claim}{Claim}{Claims}

\newtheorem*{claim*}{Claim}

\crefname{subclaim}{Subclaim}{Subclaims}

\newtheorem*{subclaim*}{Subclaim}

\newenvironment{case*}[1]{\smallskip\par\noindent \textit{Case}:~#1.\rmfamily}{}

\crefname{notation}{Notation}{Notations}
\newtheorem{notation}[equation]{Notation}
\newtheorem*{notation*}{Notation}

\newtheorem{terminology}[equation]{Terminology}
\newtheorem*{terminology*}{Terminology}

\crefname{convention}{Convention}{Conventions}

\newtheorem*{convention*}{Convention}
\newtheorem*{conventions*}{Conventions}

\crefname{spec}{Speculation}{Speculations}

\newtheorem*{spec*}{Speculation}

\crefname{caution}{Caution}{Cautions}

\newtheorem*{caution*}{Caution}

\crefname{hypothesis}{Hypothesis}{Hypotheses}

\newtheorem*{hypothesis*}{Hypothesis}

\crefname{assumption}{Assumption}{Assumptions}

\newtheorem*{assumption*}{Assumption}

\newcommand{\fntsz}[1][11]{\fontsize{#1}{#1}\selectfont}

\newenvironment{acknowledgements}[1][11]{\medskip \fntsz[#1]\begin{trivlist}
		\item[\hskip \labelsep {\textit{Acknowledgements}.}]}{\end{trivlist}\smallskip}

\crefname{examples}{Examples}{Examples}

\newenvironment{examples*}[1][\alph*]
{
	\refstepcounter{equation}
	\medskip
	\noindent\textbf{Examples.}
	\medskip
	\begin{enumerate}[\bfseries(\theequation.#1),ref=(\theequation.#1),itemsep=5pt]
}
{
	\end{enumerate}
	\smallskip
}

\theoremstyle{remark}

\declaretheoremstyle[
spaceabove=\topsep, 
spacebelow=6pt,
headfont=\normalfont\itshape,
notefont=\normalfont, notebraces={(}{)},
bodyfont=\normalfont,
postheadspace=4pt,
qed=\mbox{\smaller[4]$\boxtimes$}
]{claimproofstyle}
\declaretheorem[name={Proof of Claim}, style=claimproofstyle, unnumbered]{pf}

\crefname{subsection}{Subsection}{Subsections}

\crefformat{footnote}{#2\footnotemark[#1]#3}

\crefrangeformat{enumi}{#3#1#4--#5#2#6}

\theoremstyle{plain}

\usepackage{subfiles}

\usepackage{mdframed}
\newmdenv[
leftmargin = 1cm,
rightmargin = 0pt,
skipabove = 8pt,
skipbelow = 3pt,
innerleftmargin = 8pt,
innertopmargin = 0pt,
innerbottommargin = 0pt,
innerrightmargin = 0pt,
linewidth = 3pt,
topline = false,
rightline = false,
bottomline = false
]{leftbar}

\definecolor{gris}{RGB}{90,90,90}

\definecolor{vert}{RGB}{7,126,26}

\definecolor{purple}{RGB}{116,0,159}

\makeatletter
\def\@settitle{\begin{center}%
		\baselineskip14\p@\relax
		\bfseries
		\uppercasenonmath\@title
		\@title
		\ifx\@subtitle\@empty\else
		\\[1ex]\uppercasenonmath\@subtitle
		\footnotesize\mdseries\@subtitle
		\fi
	\end{center}%
}
\def\subtitle#1{\gdef\@subtitle{#1}}
\def\@subtitle{}
\makeatother

\makeatletter
\def\l@section{\@tocline{1}{5pt}{0pc}{}{}}
\renewcommand{\tocsection}[3]{%
	\indentlabel{\@ifnotempty{#2}{\makebox[20pt][l]{%
				\ignorespaces#1 #2.\hfill}}}\sc #3\dotfill}

\newdimen{\tocsubsecmarg}
\def\l@subsection{\@tocline{2}{3pt}{0pc}{\tocsubsecmarg}{}}
\renewcommand{\tocsubsection}[3]{%
	\indentlabel{\@ifnotempty{#2}{\makebox[30pt][l]{%
				\ignorespaces#1 #2.\hfill}}}#3\dotfill}
\makeatother
\let\oldtocsubsection=\tocsubsection
\renewcommand{\tocsubsection}[2]{\hspace{3em} \oldtocsubsection{#1}{#2}}

\newcommand{\stoptocwriting}{%
	\addtocontents{toc}{\protect\setcounter{tocdepth}{-5}}}
\newcommand{\resumetocwriting}{%
	\addtocontents{toc}{\protect\setcounter{tocdepth}{\arabic{tocdepth}}}}

\title{A descriptive construction of trees and Stallings' theorem}
\author{Anush Tserunyan}

\address[Anush Tserunyan]{Department of Mathematics, University of Illinois at Urbana-Champaign, IL, 61801, USA}
\email{anush@illinois.edu}
\thanks{The author's research was partially supported by NSF Grant DMS-1501036.}

\newcommand{\edgec}[2][\SC]{ \big( [#2]_{#1}, [#2^c]_{#1} \big) }
\newcommand{\edgeC}[2][\SC]{ \big( [#2]_{#1}, [(#2)^c]_{#1} \big) }

\newcommand{\bd}{\del}
\newcommand{\cobd}{\de}

\newcommand{\Cthin}{\hat{\CC}}

\newcommand{\NNG}{\NC}
\newcommand{\NNGk}[1][k]{\NNG_k}
\newcommand{\nd}[1][k]{{d_{#1}}}

\begin{document}

\begin{abstract}
	We give a descriptive construction of trees for multi-ended graphs, which yields yet another proof of Stallings' theorem on ends of groups. Even though our proof is, in principle, not very different from already existing proofs and it draws ideas from \cite{Kron:Stallings}, it is written in a way that easily adapts to the setting of countable Borel equivalence relations, leading to a free decomposition result and a sufficient condition for treeability.
\end{abstract}

\maketitle

\tableofcontents

\section{Introduction}

Stallings' theorem on ends of groups equates (as most impressive results in geometric group theory do) a geometric property of the Cayley graph of a finitely generated (f.g.) group with a structural/algebraic characterization of the group. Call a f.g. group $\Ga$ \emph{multi-ended} if for some (equivalently, any) finite generating subset $F \subseteq \Ga$, the Cayley graph $G$ induced by $F$ has more than one \emph{end}, that is: there is a finite set of edges of $G$, removing which results in at least two infinite connected components. We say that a group $\Ga$ \emph{splits} over its subgroup $\De < \Ga$ if either $\Ga = \ast_\De K$ (HNN-extension) or $\Ga = K \ast_\De \La$ (amalgamated product). The following theorem was proven in \cite{Stallings:ends:torsion-free} for torsion-free groups and in \cite{Stallings:ends:general} for general f.g. groups.

\begin{theorem}[Stallings 1968--71]
	A f.g. group $\Ga$ is multi-ended if and only if it splits over a finite subgroup $\De < \Ga$.
\end{theorem}

There are many proofs of this theorem and in this paper, we give yet another one, which, however, is adaptable to the Borel context when working with countable\footnote{An equivalence relation is said to be \emph{countable} if each of its classes is countable.} Borel equivalence relations on standard Borel spaces.

The main statement we prove is the following structural result, which implies Stallings' theorem via Bass--Serre theory:

\begin{theorem}\label{intro:action_on_multiended_graph==>action_on_tree}
	If a group $\Ga$ admits a transitive action on a connected multi-ended graph $G$ with all vertex-stabilizers being finite, then it also admits an action on a tree with finite edge-stabilizers and without fixed points.
\end{theorem}

The latter implies a slight strengthening of Stallings' theorem:

\begin{theorem}\label{intro:general_Stallings}
	For a f.g. group $\Ga$, the following are equivalent:
	
	\begin{enumerate}[(1)]
		\item $\Ga$ is multi-ended.
		
		\item $\Ga$ admits a transitive action on a multi-ended graph $G$ with all vertex-stabilizers being finite.
		
		\item $\Ga$ splits over a finite subgroup $\De$.
	\end{enumerate}
\end{theorem}


Our proof of \cref{intro:action_on_multiended_graph==>action_on_tree} is based on Kr\"{o}n's slick construction of a nested family of cuts in $G$ that is invariant under the action of $\Ga$ \cite{Kron:Stallings}*{Theorem 3.3}, however our construction of the tree on this family is different. It is this construction that adapts to countable Borel equivalence relations, yielding \cref{intro:Borel_Stallings}, before stating which, we roughly define and explain the involved objects.

For a graph $G$ on a set $X$, a \emph{cut} is an infinite set $C \subseteq X$ contained in a single connected component $Y$ of $G$ such that $Y \setminus C$ is also infinite but there are only finitely many edges of $G$ between $C$ and $Y \setminus C$. If $G$ is a locally countable Borel graph (i.e. $G \subseteq X^2$ is a Borel set) on a standard Borel space $X$, then the set $\CC_G$ of all cuts is also naturally a standard Borel space. Call a set $\CC \subseteq \CC_G$ \emph{$G$-complete} if it contains at least one cut from every connected component of $G$.

The collection $\CC_G$ admits $G$-complete Borel subsets $\CC \subseteq \CC_G$ with certain desired properties (being self dual, nested, chain-vanishing, and meeting every $G$-connected component) and we temporarily call such $\CC$ \emph{good}. Good collections are readily available, e.g. $\Cthin'_G$, as defined below right before \cref{C''_is_nested}, or any Borel maximal non-nested subset of the set of thin cuts, see \cref{remark:maximal_nonnested_collection}. On a collection $\CC$, we define a Borel binary relation $\sim_\CC$ (\cref{defn:sim}), which turns out to be an equivalence relation if $\CC$ is good.

\begin{theorem}[Stallings for equivalence relations]\label{intro:Borel_Stallings}
	Let $E$ be a countable Borel equivalence relation on a standard Borel space $X$ and let $G$ be a multi-ended Borel graphing of $E$. For any $G$-complete good Borel collection $\CC \subseteq \CC_G$, there is a treeable\footnote{That is: admits an acyclic Borel graphing.} equivalence relation $E_T$ and a Borel equivalence relation $E_\CC \le_B \; \sim_\CC$ such that $E = E_T \ast E_\CC$.
\end{theorem}

Here, by $E = E_T \ast E_\CC$, we mean that $E$ is the free product of $E_T$ and $E_\CC$ as introduced in \cite{Gaboriau:cost}*{Subsection IV-B}, and $E_\CC \le_B \; \sim_\CC$ means that $E_\CC$ is \emph{Borel reducible} to $\sim_\CC$, i.e., there is a Borel map $\pi : X \to \CC$ such that for any $x,y \in X$,
\[
x E_\CC y \iff \pi(x) \sim_\CC \pi(y).
\] 
The precise statement of \cref{intro:Borel_Stallings} is given in \cref{Borel_Stallings}.

\cref{intro:Borel_Stallings} can be used to prove treeability of some equivalence relations $E$ via isolating suitable Borel collections $\CC$, to which the following applies:

\begin{cor}[Condition for treeability]\label{intro:treeability_criterion}
	Let $E$ be a countable Borel equivalence relation on a standard Borel space $X$. If $E$ admits a multi-ended Borel graphing $G$ and a $G$-complete good Borel collection $\CC \subseteq \CC_G$ such that the equivalence relation $\sim_\CC$ is treeable (in particular, if it is smooth or hyperfinite), then $E$ is treeable.
\end{cor}

\begin{acknowledgements}
	I thank Institut Mittag-Leffler (Sweden) and the organizers of the program ``Classification of Operator Algebras: Complexity, Rigidity, and Dynamics'' in Spring of 2016 as the present research was done within this program at the institute. I am most grateful to Damien Gaboriau for encouraging this line of thought and for his feedback. I also thank Clinton Conley, Andrew Marks, and Robin Tucker-Drob for going through my construction with me, which improved by understanding of it.
\end{acknowledgements}

\stoptocwriting
\subsection*{Comparison with other results and proofs}


Our proof of \cref{intro:action_on_multiended_graph==>action_on_tree} is not, in principle, too different from other proofs existing in the literature, e.g., \cite{Dunwoody:structure_trees}, \cite{Dicks-Dunwoody}, and \cite{Kron:Stallings}, in the sense that it uses some of the common ideas involved in constructions of trees such as nested sets (see \cref{subsec:nested-sets}) and thin cuts (see \cref{subsec:thin-neat_cuts}), as well as the equivalence relation in \cref{defn:sim}. 


\begin{figure}[ht]
	\centering		
	\includegraphics[scale=.5]{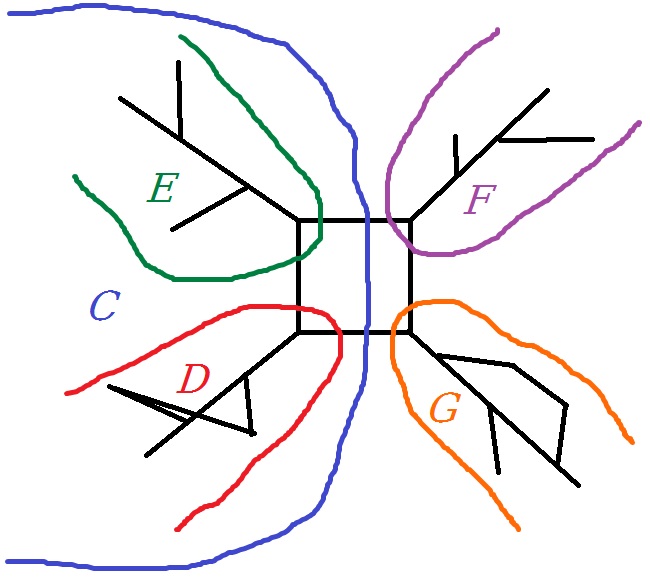}
	\caption{Nested collection $\CC \defeq \set{C,D,E,F,G}$ of thin cuts such that $C \notin \CC(B)$ for every $\CC$-block $B$.}\label{fig:counterexample-for-Kron}
\end{figure}

The shortest proof of a version of \cref{intro:action_on_multiended_graph==>action_on_tree} that the current author is aware of is presented in \cite{Kron:Stallings}, featuring a slick construction of a nested collection of cuts invariant under the action of the group \cite{Kron:Stallings}*{Sections 2 and 3} and a simpler construction of a tree using blocks \cite{Kron:Stallings}*{Section 4}. Trying to understand \cite{Kron:Stallings}*{Section 4} is what initiated the present research because there seem to be issues in the very definition of the tree. More precisely, the existence of a block $B_C$ for each $C \in \CC$ claimed in \cite{Kron:Stallings}*{Lemma 4.1} is false and \cref{fig:counterexample-for-Kron} depicts a counterexample. This invalidates the claim of \cite{Kron:Stallings}*{Theorem 4.2} that the defined graph $T(\CC)$ is a tree, namely, it can be disconnected. Also, even if \cite{Kron:Stallings}*{Lemma 4.1} was true, the proof of \cite{Kron:Stallings}*{Theorem 4.2} (both parts: acyclicity and connectedness) seems oversimplified and does not make sense to the present author. 
The incorrectness of some parts of \cite{Kron:Stallings} is also mentioned in \cite{Hensel-Kielak}*{Footnote 1 on page 4}.

Thus, the proof of \cref{intro:action_on_multiended_graph==>action_on_tree} given in the current paper is the simplest one known to the present author: it combines the construction of the non-nested collection given in \cite{Kron:Stallings}*{Sections 2 and 3} and a construction of a tree that adapts to the setting of countable Borel equivalence relations. To keep the present paper self-contained, \cite{Kron:Stallings}*{Sections 2 and 3} are rewritten in \cref{subsec:minimizing_nonnestedness} in the local terminology.

There are a number of other related results and proofs revolving around similar ideas appear, see, for example, \cites{Dunwoody:accessibility_and_groups,Dicks-Dunwoody,Dunwoody:cutting_graphs,Dunwoody-Kron:vertex_cuts,Evangelidou-Papasoglu:cactus,Kron:quasi-isometries,Moller:ends_of_graphs,Moller:ends_of_graphs_II,Hensel-Kielak}. We refer the reader to \cite{Kron:Stallings}*{Section 5} for a concise description and comparison of some of these results.

As for \cref{intro:Borel_Stallings}, a similar free decomposition result was proven in \cite{Ghys:topologie_feuilles} via different methods. The most relevant statement in the latter paper is that if a graphing of a countable Borel equivalence relation $E$ has infinitely-many ends in each connected component then the equivalence relation decomposes into a free product of a nontrivial hyperfinite subequivalence relation and some other equivalence relation $E'$. It does not, however, give a direct insight into the nature/structure of $E_1$ as \cref{intro:Borel_Stallings} does with $E_\CC$, and the proof is significantly longer.

Ends of Borel graphs have also been extensively studied, from a different angle, in \cite{Miller_thesis}, \cite{Miller:ends_of_graphs_I}, and \cite{Hjorth-Miller:ends_of_graphs_II}.

\stoptocwriting
\subsection*{Organization}

\cref{sec:construction_of_tree} describes a construction of a tree on an abstract collection of sets satisfying certain properties. This is then applied to collections of cuts in multi-ended graphs in \cref{sec:application_to_ends_of_graphs}, yielding \cref{intro:action_on_multiended_graph==>action_on_tree,intro:general_Stallings}. Finally, \cref{sec:Borel-Stallings} is where we discuss the context of countable Borel equivalence relation and prove \cref{intro:Borel_Stallings}.

\resumetocwriting

\section{Constructing a tree on a collection of sets}\label{sec:construction_of_tree}

Fix an ambient set $X \ne \0$ and, henceforth, by a \emph{set} we will mean a subset of $X$. Let $\SC$ always denote a collection of \textbf{nonempty} subsets of $X$.

For a set $A \subseteq X$, we denote by $A^c$ the complement of $A$ within $X$ and for $i \in \set{-1,1}$, we put
\[
A^i \defeq \begin{cases}
A & \text{if } i = 1\\
A^c & \text{if } i = -1.
\end{cases}
\]

\subsection{Orthogonality, domination, and basis}

\begin{defn}
	Call sets $A,B$ \emph{orthogonal} and write $A \perp B$ if $A \cap B = \0$ and $A \ne B^c$. On the other hand, say that $A$ (resp., \emph{exactly}) \emph{dominates} $B$ if $A \supseteq B^i$ for some $i \in \set{-1,1}$ (resp., $A \supseteq B$). We say that a collection $\AC$ of sets (resp., \emph{exactly}) \emph{dominates} $B$ if some set in $\AC$ (resp., exactly) dominates $B$.
\end{defn}

\begin{obs}\label{orthogonal==>not_dominating}
	For sets $A,B$, if $A \perp B$, then neither dominates the other. In particular, $A^i \nperp B^j$ for any $(i,j) \in \set{-1,1}^2$ other than $(i,j) = (1,1)$.
\end{obs}

\begin{defn}
	A collection $\BC$ of sets is called \emph{orthogonal} if any two distinct sets in $\BC$ are orthogonal. We say that $\BC$ \emph{dominates} a collection of sets $\SC$ if every set in $\SC$ is dominated by some set in $\BC$. We call $\BC \subseteq \SC$ a \emph{basis} for $\SC$ if it is orthogonal and dominates $\SC$.
\end{defn}

Note that a basis can be infinite as well as finite, and we give concrete examples of both in \cref{remark:infinite-and-finite_bases} for collections of cuts in a graph.

Our global goal is to define a special tree on the set $\BF(\SC)$ of all bases for $\SC$. However, at this point, even the existence of a basis is not clear and our local goal is to prove it under some hypotheses on $\SC$.

\medskip

We continue by recording some easy properties of bases.

\begin{obs}\label{basis_no_proper_extension}
	For a collection of sets $\SC$ and $\BC_0 \subseteq \BC_1 \subseteq \SC$, if both $\BC_0, \BC_1$ are bases for $\SC$, then $\BC_0 = \BC_1$.
\end{obs}
\begin{proof}
	Immediate from \cref{orthogonal==>not_dominating}
\end{proof}

\begin{lemma}\label{orthogonal-to-basis==>exactly_dominated}
	For a collection $\SC$ of sets with a basis $\BC \subseteq \SC$, if a set $A \in \SC$ is orthogonal to a set $B \in \BC$, then $A$ is exactly dominated by $\BC$.
\end{lemma}
\begin{proof}
	Since $\BC$ dominates $A$, there is $C \in \BC$ with $C \supseteq A^i$, for some $i \in \set{-1,1}$, and $C \ne B$, by \cref{orthogonal==>not_dominating}. If $i = -1$, then $B \perp A$ gives $B \subseteq A^c \subseteq C$, contradicting $B \perp C$, so $i = 1$.
\end{proof}

\subsection{Maximally orthogonal sets}

Here, we define a tool for building bases.

\begin{defn}
	For a collection $\SC$ of sets and sets $A, B \in \SC$, say that $B$ is \emph{$\SC$-maximally orthogonal to $A$} if it is an inclusion-maximal set in $\SC$ that is orthogonal to $A$, i.e. $B \perp A$ and for any $C \in \SC$, $C \perp A$ and $C \supseteq B$ implies $C = B$. Put 
	\[
	[A]_\SC \defeq \set{A} \cup \set{B \in \SC : B \text{ $\SC$-maximally orthogonal to } A}.
	\]
\end{defn}

\begin{lemma}\label{basis_contains_classes}
	For a collection $\SC$ with a basis $\BC \subseteq \SC$, for any $A \in \BC$, $[A]_\SC \subseteq \BC$.
\end{lemma}
\begin{proof}
	Fix $B \in [A]_\SC \setminus \set{A}$, so $B \perp A$. By \cref{orthogonal-to-basis==>exactly_dominated}, there is $C \in \BC$ distinct from $A$ such that $C \supseteq B$. But $C \perp A$, so by the maximality of $B$, $C = B$. 
\end{proof}

We now introduce a condition on $\SC$, which ensures that $[A]_\SC$ contains enough sets.

\begin{defn}
	A sequence $(A_n)_{n \in \N}$ of sets is called a \emph{chain} if it is strictly monotone (i.e. either $\subsetneq$-increasing or $\subsetneq$-decreasing). Call the chain $(A_n^c)_{n \in \N}$ the \emph{dual} of $(A_n)_{n \in \N}$. For a collection $\SC$ of sets, call a decreasing (resp. increasing) chain $(A_n)_{n \in \N}$ \emph{$\SC$-vanishing} if there is no $B \in \SC$ such that $A_n \supseteq B$ (resp. $A_n \cap B = \0$) for every $n \in \N$. A collection of sets $\SC$ is said to be \emph{chain-vanishing} if every chain in it is $\SC$-vanishing.
\end{defn}

In \cref{example:1-edge_cuts} below, we provide an instance of a chain-vanishing collection of sets.

\begin{lemma}\label{chain-vanishing==>exists_maximally_orthogonal}
	For a chain-vanishing collection $\SC$ and any $A,B \in \SC$, if $A \perp B$, then there is $C \in [A]_\SC$ with $C \supseteq B$.
\end{lemma}
\begin{proof}
	Otherwise, we contradict the chain-vanishing property by recursively building an increasing chain $(B_n)_{n \in \N}$ with $B_0 = B$ and $B_n \perp A$ for every $n \in \N$. Indeed, assume $(B_i)_{i \le n}$ is already defined and $B_n \perp A$. We know that $B_n \notin [A]_\SC$, so it is not maximal in $\SC$ among the sets orthogonal to $A$, and thus, there is $B_{n+1} \in \SC$ with $B_{n+1} \supseteq B_n$ and $B_{n+1} \perp A$.
\end{proof}

\subsection{Nested sets}\label{subsec:nested-sets}

\begin{terminology}
	For sets $A,B$, the sets $A^i \cap B^j$, $i,j = \pm 1$, are called the \emph{corners} of $(A,B)$.
\end{terminology}

\begin{defn}
	Sets $A$ and $B$ are called \emph{nested} if they have an empty corner. A collection of sets $\SC$ is called \emph{nested} if any two sets in $\SC$ are nested.
\end{defn}

\begin{example}\label{example:1-edge_cuts}
	Let $\TC \defeq (X,E)$ be an acyclic graph on $X$ and let $\SC$ be the collection of all subsets of $X$ that are connected components of the graph obtained from $\TC$ by removing a single edge, i.e.
	\[
	\SC \defeq \set{C \subseteq X : C \text{ is a connected component of } \TC - \set{e}, e \in E}.
	\]
	It follows from acyclicity of $\TC$ that $\SC$ is nested. Moreover, if $\TC$ is a tree, then $\SC$ is also chain-vanishing.
\end{example}

Using orthogonality and domination, we rephrase nestedness as a (nonexclusive) alternative.

\begin{obs}[Alternative for nested sets]\label{alternative_for_pair}
	For any nested sets $A,B$, either $A$ dominates $B$ or $A \perp B^j$, for some $j \in \set{-1,1}$.
\end{obs}
\begin{proof}
	$A \supseteq B^i$ is the same as $A^c \cap B^i = \0$. Thus, if $A$ doesn't dominate $B$, then, by nestedness, it must be that $A \cap B^j = \0$ for some $j \in \set{-1,1}$.
\end{proof}

For a set $A \in \SC$, if $A$ belongs to a basis then \cref{basis_contains_classes} implies that $[A]_\SC$ is an orthogonal family. However, we still haven't shown that every $A \in \SC$ belongs to a basis. In fact, we show the orthogonality of $[A]_\SC$ first in order to show the existence of a basis.

\begin{lemma}\label{maximally_orthogonal_to_third==>orthogonal}
	For any nested collection of sets $\SC$ and $A \in \SC$, $[A]_\SC$ is an orthogonal family.
\end{lemma}
\begin{proof}
	Because $A \perp B$ and $A \perp C$, $B^c \cap C^c \supseteq A \ne \0$, which, in particular, implies that $B^c \ne C$. Also, by the maximality of $B$ and $C$, $B \nsubseteq C$ and $C \nsubseteq B$, or equivalently, $B \cap C^c \ne \0$ and $B^c \cap C \ne \0$. Thus, by nestedness, it must be that $B \cap C = \0$, so $B \perp C$.
\end{proof}

\begin{prop}\label{class_is_basis}
	For any nested, chain-vanishing collection $\SC$ of sets and any $A \in \SC$, $[A]_\SC$ is a basis for $\SC$.
\end{prop}
\begin{proof}
	By \cref{maximally_orthogonal_to_third==>orthogonal}, we only need to show that $[A]_\SC$ dominates $\SC$, so take $B \in \SC$. By \cref{alternative_for_pair}, either $A$ dominates $B$, in which case we are done, or $B^j \perp A$ for some $j \in \set{-1, 1}$. In the latter case, by the chain-vanishing property (\cref{chain-vanishing==>exists_maximally_orthogonal}), there is $C \in [A]_\SC$ with $C \supseteq B^j$.
\end{proof}

\begin{cor}\label{basis=class}
	For any nested, chain-vanishing collection $\SC$ of sets, the bases for $\SC$ are precisely the sets of the form $[A]_\SC$ for $A \in \BC$. In particular, for any $A,B \in \BC$, $[A]_\SC = [B]_\SC$, and two distinct bases are disjoint.
\end{cor}
\begin{proof}
	Immediate from \cref{basis_contains_classes,class_is_basis,basis_no_proper_extension}.
\end{proof}

\begin{defn}\label{defn:sim}
	For $A,B \in \SC$, we write $A \sim_\SC B$ if $A \in [B]_\SC$.
\end{defn}

\begin{cor}\label{nested-and-chain-vanishing=>eq-rel}
	For any nested, chain-vanishing collection $\SC$ of sets, $\sim_\SC$ is an equivalence relation on $\SC$.
\end{cor}

\subsection{The graph $\TC_\SC$}\label{subsec:tree}

Throughout this subsection, let $\SC$ be a nested chain-vanishing collection of sets. We define an undirected graph $\TC_\SC$ on $\SC / \sim_\SC$ by putting an edge $\edgec{A}$ whenever both $A,A^c \in \SC$. Note that this is an undirected (i.e. symmetric) graph and all of the graph terminology used below is in the sense of undirected graphs.

\begin{obs}\label{no_loops_multi-edges}
	$\TC_\SC$ has no loops or multi-edges, i.e. for any $A,B \in \SC$,
	\begin{enumerate}
		\item\label{item:no_loops} $[A^c]_\SC \ne [A]_\SC$;
		\item\label{item:multi-edges} $[A]_\SC = [B]_\SC$ and $[A^c]_\SC = [B^c]_\SC$ implies $A = B$.
	\end{enumerate}
\end{obs}
\begin{proof}
	\labelcref{item:no_loops} follows from the fact that $A \nperp A^c$ and \labelcref{item:multi-edges} from \cref{orthogonal==>not_dominating}.
\end{proof}

\begin{defn}
	Let $\TC_\SC$ be as above and let $\la \in \N \cup \set{\w}$. We say that a (finite or infinite) sequence $(A_n)_{n < \la} \subseteq \SC$ \emph{represents} a path if $[A_0]_\SC, [A_1]_\SC, [A_2]_\SC, \ldots$ is a path in $\TC_\SC$ and $A_{n-1} \sim_\SC A_n^c$ for each $1 \le n < \la$.
\end{defn}

\begin{lemma}\label{no_backtracking_equivalence}
	A path in $\TC_\SC$ represented by $(A_n)_{n < \la} \subseteq \SC$, $\la \ge 3$, has no backtracking\footnote{A (finite or infinite) path $(v_n)_{n < \la}$ in an undirected graph is said to have \emph{backtracking} if $v_{n-2} = v_n$ for some $2 \le n < \la$.} if and only if $A_{n-1}^c \ne A_n$ for each $1 \le n < \la$.
\end{lemma}
\begin{proof}
	For each $2 \le n < \la$, because $A_{n-1}^c \sim_\SC A_{n-2}$, we have
	\[
	[A_{n-2}]_\SC = [A_n]_\SC \iff A_{n-1}^c \sim_\SC A_n.
	\]
	But we also have $A_{n-1} \sim_\SC A_n^c$, so \labelcref{item:multi-edges} of \cref{no_loops_multi-edges} gives
	\[
	A_{n-1}^c \sim_\SC A_n \iff A_{n-1}^c = A_n.\qedhere
	\]
\end{proof}

\begin{prop}\label{paths_in_T}
	Let $\SC$ be a nested, chain-vanishing collection of sets and let $(A_n)_{n < \la} \subseteq \SC$, $\la \in \N \cup \set{\w}$, represent a path in $\TC_\SC$ with no backtracking. Then $(A_n)_{n < \la}$ is strictly increasing.
\end{prop}
\begin{proof}
	By \cref{no_backtracking_equivalence}, $A_{n-1} \ne A_n^c$, but we also have $A_{n-1} \sim_\SC A_n^c$, so $A_{n-1} \perp A_n^c$, or equivalently, $A_{n-1} \subsetneq A_n$.
\end{proof}

\begin{cor}
	$\TC_\SC$ is acyclic.
\end{cor}
\begin{proof}
	Let $n \ge 2$ and let $(A_i)_{i \le n}$ represent a path in $\TC_\SC$ with no backtracking. Assuming that $[A_0]_\SC = [A_n]_\SC$, the sequence $(A_n, A_1, A_2, \hdots, A_n)$ still represents a path (the same one). But now \cref{paths_in_T} implies $A_n \subsetneq A_n$, a contradiction.
\end{proof}

\begin{theorem}\label{T_is_a_tree}
	For any self-dual\footnote{A collection of sets is \emph{self-dual} if it is closed under complements.}, nested, chain-vanishing collection of sets $\SC$, the graph $\TC_\SC$ is a tree.
\end{theorem}
\begin{proof}
	By the previous proposition, we only need to show connectedness, so fix distinct $[A]_\SC, [B]_\SC \in \SC / \sim_\SC$. By nestedness, we have $A^i \perp B^j$ for some $i,j \in \set{-1,1}$, and by the definition of $\TC_\SC$, it is enough to show that $[A^i]_\SC$ and $[B^j]_\SC$ are connected, so we may assume without loss of generality that $A \perp B$ to begin with. 
	
	Suppose towards a contradiction that there is no path connecting $[A]_\SC$ and $[B]_\SC$.
	
	\begin{claim*}
		There is an infinite sequence $(A_n)_{n \in \N} \subseteq \SC$ representing a path in $\TC_\SC$ and such that $A_0 = A$ and $A_n \perp B$ for all $n \in \N$.
	\end{claim*}
	\begin{pf}
		Putting $A_0 \defeq A$ we assume by induction that $(A_i)_{i \le n}$, $n \ge 0$, represents a path in $\TC_\SC$ and $A_n \perp B$. By \cref{chain-vanishing==>exists_maximally_orthogonal}, there is $C \in [A_n]_\SC$ with $C \supseteq B$. Thus, $A_{n+1} \defeq C^c$ is disjoint from $B$ and $(A_i)_{i \le n+1}$ represents a path in $\TC_\SC$. Because we assume that there is no path between $[A]_\SC$ and $[B]_\SC$, $[A_{n+1}]_\SC \ne [B]_\SC$, in particular, $A_{n+1} \ne B$, hence, $A_{n+1} \perp B$.
	\end{pf}
	
	For each $n \in \N$, $A_n \perp B \perp A_{n+1}$ implies $A_n^c \supseteq B$ and $A_{n+1} \cap B = \0$, so $A_n^c \ne A_{n+1}$. Therefore, by \cref{no_backtracking_equivalence}, the path $[A_0]_\SC, [A_1]_\SC, [A_2]_\SC, \hdots$ has no backtracking, so \cref{paths_in_T} implies that $(A_n)_{n \in \N}$ is an increasing chain, contradicting the chain-vanishing property.
\end{proof}

We now apply this theorem to an action $\Ga \actson^\al X$ of a group $\Ga$. The latter naturally induces an action of $\Ga$ on $\Pow(X)$ and, for $A \in \Pow(X)$, we denote by $\Stab_\al(A) \le \Ga$ the (setwise) stabilizer of $A$, omitting the subscript when the action is clear from the context.

\begin{cor}\label{action_on_C_gives_action_on_T}
	Let $\Ga \actson^\al X$ be an action of a group $\Ga$ on a set $X$ and let $\SC$ be a self-dual, nested, chain-vanishing collection of nonempty subsets of $X$. If $\SC$ is invariant under the action $\al$, then there is a tree $\TC_\SC$ of cardinality at most $|\SC|$, on which $\Ga$ acts such that the set of all (directed) edge-stabilizers is exactly $\set{\Stab_\al(A) : A \in \SC}$. Moreover, if the action $\al$ is such that for any $A \in \SC$, there is $\ga \in \Ga$ with $\ga \cdot_\al A \ne A$ and $\ga \cdot_\al A \cap A \ne \0$, then the action $\Ga \actson \TC_\SC$ has no fixed points.
\end{cor}
\begin{proof}
	Clearly, the action $\Ga \actson \SC$ respects complements, orthogonality, and containment, and hence also the equivalence relation $\sim_\SC$. This naturally induces an action $\Ga \actson \SC / \sim_\SC$ as well as on the set of edges of $\TC_\SC$; in other words, $\Ga$ acts on $\TC_\SC$. 
	
	For the edge-stabilizers, for $A \in \SC$, every $g \in \Stab(A)$ fixes the edge $\edgec{A}$. Conversely, if $\ga \in \Ga$ fixes the edge $\edgec{A}$, i.e. $\edgec{A} = \edgeC{\ga A}$, then \labelcref{item:multi-edges} of \cref{no_loops_multi-edges} applied to $A$ and $\ga A$ implies that $A = \ga A$, and hence, $\ga \in \Stab(A)$.
	
	Now assuming the hypothesis of the ``moreover'' part, let $[A]_\SC$ be a vertex of $\TC_\SC$ and let $\ga \in \Ga$ be as in this hypothesis. Then $\ga \cdot [A]_\SC \ne [A]_\SC$ because otherwise, $\ga \cdot_\al A \in [A]_\SC$, so either $\ga \cdot_\al A = A$ or $\ga \cdot_\al A \perp A$, a contradiction.
\end{proof}

\section{Application to ends of graphs}\label{sec:application_to_ends_of_graphs}

Throughout this section, by a \emph{graph} on a set $X$ we mean an irreflexive symmetric subset of $X^2$. Let $G$ be a connected graph on $X$.

\begin{notation}
	For sets $A,B \subseteq X$, let $\cobd_G (A,B)$ denote the set of edges incident (in $G$) to both $A$ and $B$. We write $\cobd_G A$ to denote $\cobd_G (A, A^c)$ and call this set the \emph{edge-boundary} (or \emph{coboundary}) of $A$. Below, we omit writing the subscript $G$, unless it is not clear from the context.
\end{notation}

\begin{defn}\label{defn:cut}
	A \emph{cut} in a connected graph $G$ is a subset $A \subseteq X$ such that $A$ and $A^c$ are infinite and $\cobd A$ is finite.
\end{defn}

We recall that a connected graph $G$ is said to \emph{have more than one end} if it admits a cut; we also call such a graph \emph{multi-ended}. Henceforth, we assume that $G$ has more than one end and we let $\CC(G)$ (or just $\CC$) denote the set of all cuts of $G$.

Letting $\Aut(G)$ denote the group of automorphisms of $G$, call a set of vertices or edges \emph{invariant}, if it is invariant under (i.e. setwise fixed by) the natural action of $\Aut(G)$. As the action of $\Aut(G)$ on $X$ naturally induces an action on $\Pow(X)$, we also call a collection $\SC \subseteq \Pow(X)$ \emph{invariant} if it is closed under this action of $\Aut(G)$.

Our goal is to find an invariant nested chain-vanishing subcollection of $\CC$ and we do this in two stages: first, we isolate an invariant chain-vanishing collection of cuts, and then, we restrict it further to a nested, but still invariant, subcollection.

\subsection{Thin and neat cuts}\label{subsec:thin-neat_cuts}

As our first restriction, we take the collection $\Cthin$ of all \emph{thin} cuts, which, by definition, are those cuts $C$ that minimize $|\cobd C|$, i.e. 
\[
|\cobd C| = k_0 \defeq \min_{A \in \CC} |\cobd A|.
\]
Below we show that $\Cthin$ is chain-vanishing.

Call a set $A \subseteq X$ \emph{$G$-connected} (or just \emph{connected}) if the induced subgraph $G \rest{A}$ is connected. Call a cut $A$ \emph{neat} if both $A$ and $A^c$ are connected.

\begin{lemma}\label{cuts_connectedness}
	For any cut $C$, $G \rest{C}$ has at most $|\cobd(C)|$-many connected components; in fact, letting $a$ and $b$ denote the numbers of finite and infinite connected components, respectively,
	\[
	a < |\cobd(C)| \text{ and } 1 \le b \le \frac{|\cobd(C)| - a}{k_0}.
	\]
	In particular, if $C$ is thin, then it is connected. Thus, thin cuts are neat.
\end{lemma}
\begin{proof}
	The main observation is that for each connected component $A \subseteq C$ of $G \rest{C}$, $\cobd A \subseteq \cobd C$, so $\cobd C$ is the (disjoint) union of the $\cobd A$ with $A$ ranging over the connected components of $G \rest{C}$. If $A$ is infinite, then it is a cut, so $|\cobd A| \ge k_0$, and hence the inequality above. Finally, the neatness of thin cuts follows from their closedness under complements.
\end{proof}

For $k \ge 1$, let $\CC_k$ denote the collection of neat cuts, whose edge-boundary has exactly $k$ elements. By \cref{cuts_connectedness}, $\CC_{k_0} = \Cthin$.

\begin{remark}\label{remark:infinite-and-finite_bases}
	It is worth emphasizing that bases for $\Cthin$ can be infinite as well as finite, see \cref{fig:finite-basis,fig:infinite-basis}.
	
\begin{tabular}{@{\hspace{-\parindent}}c@{\hspace{-1.5cm}}c}
	\begin{minipage}{.5\textwidth}
	\begin{figure}[H]
			\includegraphics[scale=.5]{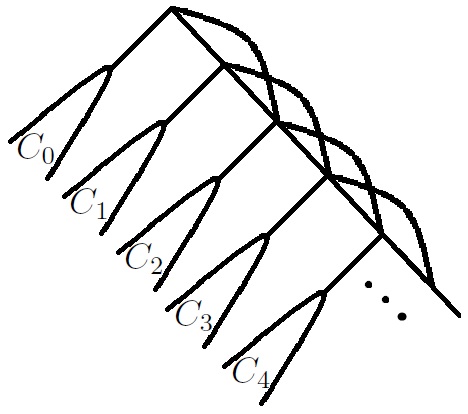}
		\caption{An infinite basis $\set{C_0,C_1,C_2,\hdots}$ for $\Cthin = \CC_{k_0}$ with $k_0 = 1$.}\label{fig:infinite-basis}
	\end{figure}
	\end{minipage}
	&
	\begin{minipage}{.53\textwidth}
		\begin{figure}[H]
			\includegraphics[scale=.5]{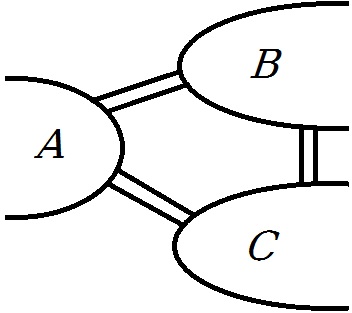}
			\caption{A finite basis $\set{A,B,C}$ for $\Cthin = \CC_{k_0}$ with $k_0 = 2$.}\label{fig:finite-basis}
		\end{figure}
	\end{minipage}
\end{tabular}
\end{remark}

\begin{prop}\label{C_k_has_chain-vanishing}
	For any $k \ge 1$, $\CC_k$ is chain-vanishing. In particular, $\Cthin$ is chain-vanishing.
\end{prop}
\begin{proof}
	Because $\CC_k$ is closed under complements, it is enough to prove that there are no proper decreasing chains. Assuming towards a contradiction that $(C_n)_{n \in \N} \subseteq \CC_k$ is a strictly decreasing chain with $C_\w \defeq \bigcap_{n \in \N} C_n \ne \0$, we recursively build a sequence $(e_m)_{m \in \N}$ of pairwise distinct edges of $G$ such that for each $m \in \N$, $e_m \in \cobd C_n$ for all large enough $n \in \N$. Granted such a sequence, we get $n \in \N$ with $\cobd C_n \supseteq \set{e_0, e_1, \ldots, e_k}$, contradicting $|\cobd C_n| = k$.
	
	Suppose by induction that a desired sequence $(e_m)_{m < \ell}$ of length $\ell \ge 0$ has already been constructed. Thus, there is $n$ large enough such that $\cobd C_n \supseteq \set{e_m}_{m < \ell}$. By the strictness of our chain, $C_n \setminus C_\w \ne \0$. Because $C_n$ is connected, there is a path connecting the set $C_n \setminus C_\w$  to $C_\w$ (i.e. a vertex in one to a vertex in the other). Consequently, there exists $e_\ell \in \cobd(C_n \setminus C_\w, C_\w)$; in particular, $e_\ell \notin \cobd C_n \supseteq \set{e_m}_{m < \ell}$. Let $\set{u,v} = e_\ell$ with $u \in C_n \setminus C_\w$ and $v \in C_\w$, and let $m > n$ be large enough such that $u \notin C_m$. Because $v \in C_m$ and $u \in C_m^c$, $e_\ell \in \cobd C_m$, concluding the recursive construction.
\end{proof}

\subsection{Minimizing the degree of non-nestedness}\label{subsec:minimizing_nonnestedness}

This subsection is almost entirely taken from \cite{Kron:Stallings}*{Sections 2 and 3} and simply rewritten here in our terminology for the sake of keeping the paper self-contained.

Note that for any $k \ge 1$, $\CC_k$ is invariant, so, by \cref{C_k_has_chain-vanishing}, we could take any $\CC_k$ as our first restriction, as long as it is nonempty. Now consider the non-nestedness graph $\NNG$ on $\Pow(X)$, i.e.
\[
\set{A,B} \in \NNG \defequiv \text{$A$ and $B$ are not nested}.
\]
The nested subcollections of $\CC_k$ are exactly the $\NNG$-independent ones, and our goal is to find one that is invariant. \cref{non-nestedness_is_locally_finite} below implies that $\NNGk \defeq \NNG \rest{\CC_k}$ is locally finite, and we take as a candidate the subcollection $\CC_k' \subseteq \CC_k$ of all vertices with minimum $\NNGk$-degree. This is clearly invariant, but may not be $\NNG$-independent in general. However, when $k$ is minimum (i.e. $k = k_0$, so $\CC_k = \Cthin$), $\CC_k'$ turns out to indeed be $\NNG$-independent. The requirement of $k$ being minimum is used through \cref{opposite_corners:thin_cuts} below and is essential for the argument.

\subsubsection{The local finiteness of $\NNGk$}

\begin{defn}
	Call a set of edges $F \subseteq G$ an \emph{edge-cut} if $F = \cobd A$ for some cut $A$.
\end{defn}

\begin{obs}\label{neat_cut=minimal_sep}
	A cut $A$ is neat if and only if $\cobd A$ is a minimal edge-cut, i.e. no proper subset of $\cobd A$ is an edge-cut.
\end{obs}

\begin{lemma}\label{fin-many_separators_containing_edge}
	For every $k \ge 1$, each edge $e \in G$ belongs to only finitely many minimal edge-cuts of size $k$.
\end{lemma}
\begin{proof}
	We prove this by induction on $k$. The base case $k=1$ is obvious, so assume the statement is true for $k \ge 1$. If $e$ does not belong to any minimal edge-cut of size $k+1$, we are done, so suppose it does. Let $u,v$ be the vertices incident to $e$.
	
	\begin{claim*}
		There is a path $P$ in $G - \set{e}$ connecting $u$ to $v$; in particular, $G \setminus \set{e}$ is connected.
	\end{claim*}
	\begin{pf}
		By \cref{neat_cut=minimal_sep} and because $k+1 \ge 2$, there is a neat cut $A$ with $\cobd A \supsetneq \set{e}$ and $u \in A$. Let $e' \defeq \set{u',v'}$ be another edge in $\cobd A$ with $u' \in A$. Because $A$ is $G$-connected, there is a path $P_A$ in $G \rest{A}$ connecting $u$ to $u'$. Similarly, there is a path $P_{A^c}$ in $G \rest{A^c}$ connecting $v'$ and $v$. Thus, the path $P_A \conc e' \conc P_{A^c}$ lies in $G \setminus \set{e}$ and connects $u$ to $v$.
	\end{pf}
	
	Every minimal edge-cut of size $k+1$ containing $e$ becomes a minimal edge-cut of size $k$ in $G - \set{e}$ and must contain at least one edge lying on the path $P$. But by the induction hypothesis applied to $G - \set{e}$, every edge lying on $P$ belongs to only finitely many minimal edge-cuts of size $k$ in $G - \set{e}$, so the fact that $P$ is finite concludes the proof.
\end{proof}

\begin{notation}
For a set $A \subseteq X$, let $N_k^G(A)$ (or just $N_k(A)$) denote the set of all cuts in $\CC_k$ that are not nested with $A$; call it the \emph{$k$-neighborhood} of $A$.
\end{notation}

\begin{notation}
	For a set $A \subseteq X$, let $\bd_G A$ (or just $\bd A$) denote the set of vertices in $A$ that are adjacent to vertices in $A^c$ (equivalently, incident to edges in $\cobd A$); call $\bd A$ the \emph{vertex-boundary} (or just \emph{boundary}) of $A$.
\end{notation}

\begin{prop}\label{non-nestedness_is_locally_finite}
	For every $k \ge 1$ and set $A \subseteq X$ with finite edge-boundary, $N_k(A)$ is finite.
\end{prop}
\begin{proof}
	First, we relate the sets in $N_k(A)$ to vertices in $\bd A$.
	
	\begin{claim*}
		For any $B \in N_k(A)$, there are distinct vertices $v_-, v_+ \in \bd A$ such that for every path $P$ connecting them, $\cobd B$ contains an edge lying on $P$.
	\end{claim*}
	\begin{pf}
		For each $i \in \set{-1,1}$, because $B^i$ is connected and the corners $A \cap B^i$ and $A^c \cap B^i$ are nonempty, $\cobd (A \cap B^i, A^c \cap B^i) \ne \0$, so there is a vertex $v_i \in A \cap B^i$ that is incident to an edge in $(A \cap B^i, A^c \cap B^i)$. In particular, $v_i \in \bd (A) \cap B^i$.
		
		Now, $v_1 \in B$ and $v_{-1} \in B^c$, so they are distinct. Moreover, any path connecting them has to intersect $\cobd B$.
	\end{pf}
	
	Thus, for each distinct pair $u, v \in \bd A$, we fix a path $P_{u,v}$ connecting them and we let $F$ be the set of edges that lie on at least one of these paths. Because $\bd A$ is finite and for each $u, v \in \bd A$, the number of edges lying on $P_{u,v}$ is finite, $F$ is finite. By the claim, for each $B \in N_k(A)$, $\cobd B$ contains an edge from $F$. But $\cobd B$ is a minimal edge-cut of size $k$ and, by \cref{fin-many_separators_containing_edge}, each edge is contained in only finitely many such edge-cuts, so the finiteness of $F$ concludes the proof.
\end{proof}

\begin{notation}
	For a set $A \subseteq X$, let $\nd^G(A)$ (or just $\nd(A)$) denote $|N_k(A)|$ and call it the \emph{$k$-degree} of $A$.
\end{notation}

\subsubsection{Moving from $\NNG$-neighbors to their opposite corners}

\begin{terminology}
	For sets $A,B$, corners $A^{i_0} \cap B^{j_0}$ and $A^{i_1} \cap B^{j_1}$ are said to be \emph{opposite} if $i_0 i_1 = j_0 j_1 = -1$. For a corner $C$ of $(A,B)$, we denote its opposite corner by $-C$.
\end{terminology}

Our local goal is to show that moving from $\NNG$-neighbors to their opposite corners lowers the $k$-degree, provided all of the sets involved are in $\CC_k$. This will imply that the thin cuts with minimum $k_0$-degree are $\NNG$-independent, i.e. nested.

\begin{lemma}\label{N(A_cap_B)_subset_N(A)_cup_N(B)}
	For sets $A,B,C$, if $A^c \cap B^c \ne \0$ and $C$ is not nested with $A \cap B$, then $C$ is not nested with either $A$ or $B$.
\end{lemma}
\begin{proof}
	For some $i \in \set{1, -1}$, $C^i$ intersects $A^c \cap B^c$. Also, $C^{-i}$ intersects $A^c \cup B^c$. If $C^{-i}$ intersects $A^c$, then $C$ is not nested with $A$; otherwise, it is not nested with $B$.
\end{proof}

\begin{lemma}\label{opposite_corners:nestedness}
	For sets $A,B,C \subseteq X$, if there are opposite corners of $(A,B)$ such that $C$ is not nested with either of them, then $C$ is not nested with either of $A,B$.
\end{lemma}
\begin{proof}
	Because nestedness is immune to taking complements, we may assume that $C$ is not nested with $A \cap B$ and $A^c \cap B^c$. In particular, the corners $C^i \cap A$ and $C^i \cap A^c$ are nonempty for each $i \in \set{-1,1}$, so $C$ is not nested with $A$. Similarly, $C$ is not nested with $B$.
\end{proof}

\begin{lemma}\label{opposite_corners:lower_degree}
	Let $A,B \in \CC_k$. If $A \cap B$ and $A^c \cap B^c$ are also in $\CC_k$, then
	\[
	\nd(A \cap B) + \nd(A^c \cap B^c) \le \nd(A) + \nd(B),
	\]
	and the inequality is strict if $A$ and $B$ are not nested; in fact,
	\[
	\nd(A \cap B) + \nd(A^c \cap B^c) \le \nd(A) + \nd(B) - 2.
	\]		
\end{lemma}
\begin{proof}
	\cref{N(A_cap_B)_subset_N(A)_cup_N(B)} implies that any cut in $\CC_k$ that contributes exactly $1$ to the left-hand side, contributes at least $1$ to the right-hand side. Furthermore, by \cref{opposite_corners:nestedness}, any cut $C \in \CC_k$ that contributes $2$ to the left-hand side, also contributes $2$ to right-hand side, so the first inequality follows.
	
	If $A$ and $B$ are not nested, then $A \in N_k(B)$ and $B \in N_k(A)$, so they each contribute $1$ to the right-hand side. On the other hand, both $A$ and $B$ are nested with any of their corners, so $A,B$ contribute nothing to the left-hand side.
\end{proof}

Now we show that when $k = k_0$, the hypothesis of \cref{opposite_corners:lower_degree} is met, modulo taking complements.

\begin{lemma}\label{opposite_corners:thin_cuts}
	For any two thin $A,B$ and any infinite corner $C$ of $(A,B)$, if the opposite corner $-C$ is also infinite, then both $C$ and $-C$ are thin cuts.
\end{lemma}
\begin{proof}
	By taking complements if necessary, we may assume that $C = A \cap B$ and we suppose that $-C = A^c \cap B^c$ is infinite and hence a cut. Putting
	\begin{align*}
		e_B &\defeq |\cobd (A \cap B, A^c \cap B)| \\
		e_A &\defeq |\cobd (A \cap B, A \cap B^c)| \\
		e_{B^c} &\defeq |\cobd (A \cap B^c, A^c \cap B^c)| \\
		e_{A^c} &\defeq |\cobd (A^c \cap B, A^c \cap B^c)| \\
		d_{A B} &\defeq |\cobd (A \cap B, A^c \cap B^c)| \\
		d_{A B^c} &\defeq |\cobd (A \cap B^c, A^c \cap B)|,
	\end{align*}
	observe that 
	\begin{equation}\label{eq:bdries_of_A_and_B}
		e_B + e_{B^c} + d_{A B} + d_{A B^c} = |\cobd A| = |\cobd B| = e_A + e_{A^c} + d_{A B} + d_{A B^c}.
	\end{equation}
	Also, by the minimality of $|\cobd A|$,
	\begin{equation}\label{eq:corner_bdry}
			\begin{aligned}
			e_B + e_A + d_{A B} &= |\cobd (A \cap B)| \ge e_B + e_{B^c} + d_{A B} + d_{A B^c} 
			\\
			e_{B^c} + e_{A^c} + d_{A B} &= |\cobd (A^c \cap B^c)| \ge e_B + e_{B^c} + d_{A B} + d_{A B^c},
		\end{aligned}
	\end{equation}
	so $e_A \ge e_{B^c} + d_{A B^c}$ and $e_{A^c} \ge e_B + d_{A B^c}$. Plugging this into the right-hand side of \labelcref{eq:bdries_of_A_and_B} gives
	\begin{align*}
		e_B + e_{B^c} + d_{A B} + d_{A B^c} 
		&=
		e_A + e_{A^c} + d_{A B} + d_{A B^c}
		\\
		&\ge 
		e_{B^c} + d_{A B^c} + e_B + d_{A B^c} + d_{A B} + d_{A B^c} 
		\\
		&= 
		e_{B^c} + e_B + d_{A B} + 3 d_{A B^c},
	\end{align*}
	so $d_{A B^c} = 0$ and the inequality has to be equality, which, in the light of $e_A \ge e_{B^c}$ and $e_{A^c} \ge e_B$, implies $e_A = e_{B^c}$ and $e_{A^c} = e_B$. But now \labelcref{eq:bdries_of_A_and_B,eq:corner_bdry} together give $|\cobd (A \cap B)| = |\cobd A|$ and $|\cobd (A^c \cap B^c)| = |\cobd B|$, so $A \cap B$ and $A^c \cap B^c$ are also thin.
\end{proof}

Recalling that we denote by $\CC_k' \subseteq \CC_k$ the set of cuts of minimum $\NNGk$-degree, we take $\Cthin' \defeq \CC_{k_0}'$. Clearly, $\Cthin'$ is invariant, so it remains to show the following.

\begin{prop}\label{C''_is_nested}
	$\Cthin'$ is nested (i.e. $\NNG$-independent).
\end{prop}
\begin{proof}
	The invariance is clear. For nestedness, fix $A,B \in \Cthin'$. Because the sets $A,A^c,B,B^c$ are infinite, the Pigeonhole Principle implies that there is a pair of oppositve corners $C$ and $-C$ of $(A,B)$ that are both infinite. By \cref{opposite_corners:thin_cuts}, both $C$ and $-C$ are thin cuts, and by replacing one or both of $A,B$ with their complements, we may assume without loss of generality that $C = A \cap B$ and $-C = A^c \cap B^c$. Thus, \cref{opposite_corners:nestedness} applies, so if $A$ and $B$ were not nested, then the inequality in \cref{opposite_corners:nestedness} would be strict, contradicting the minimality of $\nd(A)$ and $\nd(B)$.
\end{proof}

\subsection{Stallings' theorem on ends of groups}

\begin{lemma}\label{stabilizers_of_finite_sets}
	Let $\Ga \actson^\al X$ be an action of a group $\Ga$ on a set $X$ with all point-stabilizers in the same commensurability class $\CF$ (e.g. $\CF \defeq \text{finite subgroups of } \Ga$). Then the set-stabilizers of finite subsets of $X$ also belong to $\CF$.
\end{lemma}
\begin{proof}
	Let $F \subseteq X$ be finite and let $H \defeq \Stab_\al(F)$. The action of $H$ on $F$ gives a homomorphism $\phi : H \to \Si(F)$. Then $[H : \ker(\phi)] < \w$, so $H$ is in the same commensurability class as $\ker(\phi)$. But $\ker(\phi) = \bigcap_{x \in F} \Stab_\al(x)$, so $\ker(\phi) \in \CF$ because $F$ is finite.
\end{proof}

\begin{cor}\label{stabilizers_of_cuts}
	Let $\Ga \actson^\al G$ be an action of a group $\Ga$ on a graph $G$ with finite vertex-stabilizers. Then the set-stabilizers of cuts are also finite.
\end{cor}
\begin{proof}
	For any $\ga \in \Ga$ and a cut $A$ of $G$, $\ga A$ is also a cut and $\ga (\bd A) = \bd (\ga A)$, so $\Stab(A) \subseteq \Stab(\bd A)$ and \cref{stabilizers_of_finite_sets} gives the conclusion.
\end{proof}

We are now ready to prove \cref{intro:action_on_multiended_graph==>action_on_tree}, which we state again here for the reader's convenience.

\begin{namedthm*}{\cref{intro:action_on_multiended_graph==>action_on_tree}}
	If a group $\Ga$ admits a transitive action on a connected multi-ended graph $G$ with all vertex-stabilizers being finite, then it also admits an action on a tree with finite edge-stabilizers and without fixed points.
\end{namedthm*}

\begin{proof}
	By \cref{C''_is_nested}, we can apply \cref{action_on_C_gives_action_on_T} the collection $\Cthin'$ as above and get an action on the tree $\TC_{\Cthin'}$. The edges of this tree are exactly the cuts in $\Cthin'$, so their stabilizers are exactly the set-stabilizers of cuts and \cref{stabilizers_of_cuts} concludes the proof. As for fixed points, let $[A]_{\Cthin'}$ be a vertex of $\TC_{\Cthin'}$, so $A$ is infinite but $\del A$ is finite. Take $x \in \del A$ and $y \in A \setminus \del A$ and let $\ga \in \Ga$ be such that $\ga \cdot x = y$. Then $\del (\ga \cdot A) \ne \del A$, so $\ga \cdot A \ne A$, and also, $\ga \cdot A \cap A \ni y$, so the hypothesis of the ``moreover'' part of \cref{action_on_C_gives_action_on_T} is satisfied, and thus, the action $\Ga \actson \TC_{\Cthin'}$ has no fixed points.
\end{proof}

Finally, combined with Bass--Serre theory \cite{Serre:Trees}, the last theorem gives the following famous corollary, which states the nontrivial implication of \cref{intro:general_Stallings}.

\begin{cor}[Stallings theorem]
	If a group $\Ga$ admits a transitive action $\Ga \actson^\al$ on a multi-ended graph $G$ with all vertex-stabilizers being finite and without fixed points, then $\Ga$ splits over a finite subgroup $\De$.
\end{cor}
\begin{proof}
	By \cref{intro:action_on_multiended_graph==>action_on_tree}, $\Ga$ admits an action on a tree with finite edge-stabilizers and without fixed points. By adding midpoints to the edges, if necessary, we may assume that the action is without inversions in the sense of \cite{Serre:Trees}*{Section 3.1, first paragraph}. By the fundamental theorem of Bass--Serre theory \cite{Serre:Trees}*{Theorem 13}, $\Ga$ is the fundamental group of a connected graph of groups $\YC$, where the groups are the vertex and edge stabilizers of the action $\al$.
	
	Now take an (undirected) edge $e$ of $\YC$ and denote its stabilizer by $\De$. Removal of $e$ gives a new graph of groups $\ZC \defeq \YC - \set{e}$ and the associativity of the construction of the fundamental group gives the following: if $\ZC$ is still connected, then $\Ga = \ast_\De K$, where $K$ is the fundamental group of $\ZC$; otherwise, $G = K \ast_\De \La$, where $K,\La$ are the fundamental groups of the two connected components of $\ZC$.
\end{proof}

\section{A Stallings theorem for Borel equivalence relations}\label{sec:Borel-Stallings}

Throughout this section, let $X$ be a standard Borel space and let $E$ be a countable\footnote{This means that each $E$-class is countable.} Borel equivalence relation on $X$. Recall that a Borel graphing $G$ of $E$ is a Borel graph\footnote{That is: an irreflexive and symmetric Borel subset of $X^2$} on $X$, whose connectedness equivalence relation $E_G$ is exactly $E$. 

All graph theoretic notions we have considered so far, such as multi-ended, cuts, basis, etc., were defined for a connected graph $G$. Given an arbitrary graph on $X$, we extend all these notions to $G$ by restricting our attention to considering only $E_G$-related subsets of $X$, i.e. each set we consider will be contained in one $G$-connected component. More precisely, in the definitions of these notions, we replace 

\begin{itemize}
	\item $\Pow(X)$ with $\Pow_{E_G}(X)$---the collection of all nonempty $E_G$-related subsets of $X$,
	
	\item a complement of a set $A \in \Pow_{E_G}(X)$ with $A^* \defeq [A]_{E_G} \setminus A$.
\end{itemize}

For example, we call $G$ \emph{multi-ended} if each of its connected components has more than one end. Furthermore, call $C \in \Pow_{E_G}(X)$ a \emph{cut} of $G$ if $\bd C$ is finite and $C, C^*$ are both infinite. 

Next, we identify each cut $C$ of $G$ with 
\[
(\cobd C, \cobd C^*)_G \defeq \set{(x,y) \in X^2 : x \in C \text{ and } y \in C^*},
\] 
so the set $\CC_G$ is just a Borel subset of $[G]^{<\w}$, hence a standard Borel space. The equivalence relation $E_G$ naturally extends to that on $[G]^{<\w}$ and we denote this extension by $\tilde{E_G}$. Note that if $G$ is multi-ended, then $\CC_G$ is a complete $\tilde{E_G}$-section.

We call a cut $C \in \CC_G$ \emph{thin} if it is thin for its $G$-connected component. We also define the non-nestedness graph $\NC_G$ on $\CC_G$ by putting an edge between $A,B \in \Pow_G(X)$ if $A$ and $B$ are in the same $G$-connected component and are non-nested.

Letting
\begin{itemize}
	\item $\Cthin_G$ be the set of thin cuts of $G$,
	
	\item $\Cthin_G' \subseteq \Cthin_G$ be the set of those cuts $C$ with minimum nestedness degree in the graph $\NC$ restricted to the set of thin cuts in the $G$-connected component of $C$,
\end{itemize}
it easily follows by the Luzin--Novikov uniformization theorem, these subsets are Borel. 

Furthermore, given a Borel set $\CC \subseteq \CC_G$, the binary relation $\sim_\CC$ on $\CC$, as in \cref{defn:sim}, is also Borel, again by the Luzin--Novikov uniformization theorem. Recall that if $\CC$ is nested and chain-vanishing, then \cref{nested-and-chain-vanishing=>eq-rel} says that $\sim_\CC$ is an equivalence relation on $\CC$.

\cref{intro:Borel_Stallings} is now easily derived from the constructions described above and is our main application. We restate it more precisely here, recalling that a set $\CC \subseteq \CC_G$ is called \emph{$G$-complete} if it contains at least one cut from every connected component of $G$.

\begin{theorem}[Stallings for equivalence relations]\label{Borel_Stallings}
	Let $E$ be a countable Borel equivalence relation on a standard Borel space $X$ and let $G$ be a multi-ended Borel graphing of $E$. For any $G$-complete self-dual nested chain-vanishing Borel collection $\CC \subseteq \CC_G$ (e.g. $\CC \defeq \Cthin'_G$), there is a treeable equivalence relation $E_T$ and a Borel equivalence relation $E_\CC \le_B \; \sim_\CC$ such that $E = E_T \ast E_\CC$.
\end{theorem}

Because free product of treeable equivalence relations is treeable, this immediately gives:

\begin{cor}
	If a countable Borel equivalence relation $E$ admits a multi-ended Borel graphing $G$ such that $\sim_{\Cthin'_G}$ is treeable (e.g. finite, smooth, hyperfinite), then $E$ is treeable.
\end{cor}


\begin{remark}\label{remark:maximal_nonnested_collection}
	Because the non-nestedness graph $\NNG$ on the set $\Cthin_G$ of thin cuts is locally finite by \cref{non-nestedness_is_locally_finite}, there is a Borel maximal $\NNG$-independent set $\CC \subseteq \Cthin_G$; indeed, by \cite{KST}*{Proposition 4.5} this graph admits a countable Borel coloring, using which one easily constructs a Borel maximal independent set. By maximality, $\CC$ must be self-dual and it is chain-vanishing by \cref{C_k_has_chain-vanishing}, so \cref{Borel_Stallings} applies to $\CC$.
\end{remark}


We devote the rest of this section to the proof of this theorem, so let $X$, $E$, $G$, and $\CC$ be as in its hypothesis.

The Luzin--Novikov uniformization theorem, again, gives a Borel function $\pi : X \to \CC$ that is a reduction of $E$ to $\tilde{E}$, i.e. for each point $x$, $\pi(x)$ is a cut $C \in \CC$ with $x \in [C]_E$. Let $E_\CC$ be the $\pi$-pullback of the equivalence relation $\sim_\CC$.

We now define the graph $T_\CC$ similarly to \cref{subsec:tree}, more precisely, for each $C \in \CC$, we put the edge $(C, C^*) \in \TC_\CC$ as well as its inverse. The tree $\TC_\CC$ described in \cref{subsec:tree} is actually defined on $\CC / \sim_\CC$, but it is easy to see that $T_\CC$ is a lift of $\TC_\CC$ to $\CC$. Moreover, the natural projection map $\CC \onto \CC / \sim_\CC$ induces a graph homomorphism $T_\CC \to \TC_\CC$ that is injective on (the edges of) $T_\CC$, so $T_\CC$ is acyclic. Let $\CC_0 \defeq \pi(X)$.

Next, we apply the edge sliding argument as in either \cite{JKL}*{Proposition 3.3(i)} or \cite{Gaboriau:cost} and get an acyclic graph $T_{\CC_0}$ whose projection to $\CC_0 / \sim_\CC$ is a treeing of $\tilde{E} \rest{\CC_0} / \sim_\CC$.

Let $T$ be the $\pi$-pullback of $T_{\CC_0}$, i.e. $x T y \defequiv \pi(x) T_{\CC_0} \pi(y)$. It follows from its definition that the projection $T_{E_\CC}$ of $T$ under quotient map $X \mapsto X / E_\CC$ is a treeing of $E/E_\CC$. This implies that $E = E_T \ast E_\CC$, finishing the proof of the theorem. \hfill \qed(\cref{Borel_Stallings})


\bigskip



\end{document}